\documentclass{gtpart}     
\usepackage{pinlabel}  
\usepackage{graphicx}  
\usepackage{subcaption}
\usepackage{amscd}

\usepackage{amsmath}
\usepackage{amssymb}
\usepackage{amsthm}
\usepackage{tabularx}
\usepackage{longtable}
\usepackage{booktabs}
\usepackage{tikz-cd}
\usepackage[normalem]{ulem}

\title{Relations amongst twists along Montesinos twins in the 4-sphere}

\author{David T Gay}

\author{Daniel Hartman}

\volumenumber{}
\issuenumber{}
\publicationyear{}
\papernumber{}
\startpage{}
\endpage{}
\doi{}
\MR{}
\Zbl{}
\received{}
\revised{}
\accepted{}
\published{}
\publishedonline{}
\proposed{}
\seconded{}
\corresponding{}
\editor{}
\version{}

\newtheorem{theorem}{Theorem}
\newtheorem{lemma}[theorem]{Lemma}
\newtheorem{proposition}[theorem]{Proposition}
\newtheorem{corollary}[theorem]{Corollary}

\theoremstyle{definition}
\newtheorem{definition}[theorem]{Definition}

\newtheorem{remark}[theorem]{Remark}

\def\Z{\mathbb Z}
\def\N{\mathbb N}

\newcommand{\into}{\ensuremath{\hookrightarrow}}

\newcommand{\id}{\mathop{\rm id}\nolimits}

\newcommand{\Diff}{\mathop{\rm Diff}\nolimits}

\begin{document}

\begin{abstract}    

Isotopy classes of diffeomorphisms of the $4$--sphere can be described either from a Cerf theoretic perspective in terms of loops of $5$--dimensional handle attaching data, starting and ending with handles in cancelling position, or via certain twists along submanifolds analogous to Dehn twists in dimension two. The subgroup of the smooth mapping class group of the $4$--sphere coming from loops of $5$--dimensional handles of index $1$ and $2$ coincides with the subgroup generated by twists along Montesinos twins (pairs of $2$--spheres intersecting transversely twice) in which one of the two $2$-spheres in the twin is unknotted. In this paper we show that this subgroup is in fact trivial or cyclic of order two.

%
\end{abstract}

\maketitle
%
%

\section{Introduction}

By the smooth mapping class group of a smooth manifold $X$ we mean $\pi_0(\mathop{\Diff^+}(X))$, where $\mathop{\Diff^+}(X)$ is the space of orientation preserving self diffeomorphisms of $X$. A good way to describe an element of $\pi_0(\mathop{\Diff^+(X)})$, where $X$ is $n$--dimensional, is to build an $(n+2)$--dimensional bundle $W$ over $S^1$ of $(n+1)$--dimensional cobordisms, such that each fiber $Z_t$ over $t \in S^1$ is a cobordism from $X$ to an $n$--manifold $X_t$ and such that $Z_0$ is canonically identified with the trivial cobordism $[0,1] \times X$. Then $W$ is a cobordism from $S^1 \times X$ to an $X$--bundle over $S^1$ with monodromy some interesting element of $\pi_0(\mathop{\Diff^+(X)})$. A good way to describe such a bundle $W$ is to give a $1$--parameter family of $(n+1)$--dimensional handle attaching data in $X$, starting and ending with data describing some number of pairs of handles in cancelling position (meaning that attaching spheres meet belt spheres transversely once).

In~\cite{GayS4Diffeos}, using the fact that every orientation preserving diffeomorphism of $S^4$ is pseudoisotopic to the identity, the first author showed that every element of $\pi_0(\mathop{\Diff^+}(S^4))$ can be given in this way by a $1$--parameter family of $2$--$3$--handle pairs, and that under favorable conditions (it is unclear whether these conditions might perhaps always be satisfied) such a family can be traded for a family involving a single $1$--$2$--handle pair. Here we study the subgroup of $\pi_0(\mathop{\Diff^+}(S^4))$ coming from families of $1$--$2$--handle pairs, the ``$1$--$2$--subgroup''. In~\cite{GayS4Diffeos} the first author gave a countable list of generators for the $1$--$2$--subgroup, and here we use the explicit descriptions of those generators to show that this subgroup is actually generated by a single element, and that the square of this element is trivial.

The generators mentioned above are in fact described as twists along certain submanifolds of $S^4$, so now we develop this perspective more carefully.
%
%

\begin{definition} \label{D:GeneralTwist}
 Given an embedding $f: S^1 \times \Sigma \into X$, for some closed, oriented surface $\Sigma$ and some smooth oriented $4$--manifold $X$, the {\em twist along $f$} is the isotopy class of diffeomorphisms $\tau_f$ obtained by choosing an orientation preserving embedding $[-1,1] \times S^1 \times \Sigma$ extending $f$ and performing a right-handed Dehn twist along $[-1,1] \times S^1$ and the identity along $\Sigma$. As an element of $\pi_0(\mathop{Diff^+}(X))$, $\tau_f$ is uniquely determined by the isotopy class of the embedding $f$. 
\end{definition}

%

\begin{definition} \label{D:MontesinosTwin}
 A {\em Montesinos twin} in a $4$--manifold $X$ is a pair $W=(R,S)$ of embeddings $R,S: S^2 \into X$, each with trivial normal bundle, which intersect transversely at two points. A {\em half-unknotted} Montesinos twin is one in which one of the two $2$--spheres is unknotted.
\end{definition}

As shown in~\cite{MontesinosI, MontesinosII} and discussed in~\cite{GayS4Diffeos}, the boundary $\partial \nu(W)$ of a neighborhood of a Montesinos twin $W$ in an oriented $4$--manifold $X$ is diffeomorphic to $T^3$, and if $X=S^4$ and we label the factors of $T^3$ as $S^1_l \times S^1_R \times S^1_S$, this parametrization of $\partial \nu(W)$ is canonically determined by the oriented isotopy class of $W$ up to independent orientation preserving reparametrizations of $S^1_l$, $S^1_R$ and $S^1_S$ and ambient isotopy in $X$. The $S^1_l$ factor is homologically trivial in $S^4$, while $S^1_R$ is a positive meridian for $R$ and $S^1_S$ is a positive meridian for $S$.

\begin{definition} \label{D:MontesinosTwist}
 Given a Montesinos twin $W$ in $S^4$, the {\em twist along $W$}, denoted $\tau_W$, is the twist along the embedding $S^1_l \times (S^1_R \times S^1_S) \into S^4$, as defined in Definition~\ref{D:GeneralTwist}. Let $\mathcal{M} \subset \mathcal{T}_1(S^4)$ be the subgroup of $\pi_0(\mathop{Diff^+}(S^4))$ generated by twists along Montesinos twins. Let $\mathcal{M}_0 \subset \mathcal{M}$ be the subgroup generated by twists along half-unknotted Montesinos twins $W=(R,S)$.
\end{definition}

As we will discuss below, $\mathcal{M}_0$ is precisely the $1$--$2$--subgroup of $\pi_0(\mathop{Diff^+}(S^4))$.

The following is our main result:

\begin{theorem} \label{T:M0order2}
The group $\mathcal{M}_0$ generated by twists along half-unknotted Montesinos twins is either trivial or cyclic of order two.
\end{theorem}

%

Besides the Cerf theoretic perspective identifying $\mathcal{M}_0$ as the $1$--$2$--subgroup, one can also think of $\mathcal{M}_0$ as the simplest class of ``twist subgroups'' of $\pi_0(\mathop{Diff^+}(S^4))$, with $\mathcal{M}$, the subgroup generated by twists along arbitrary Montesinos twins, being the next interesting case. Continuing from there one can consider subgroups generated by twists along general embeddings of $S^1 \times \Sigma_g$, for surfaces of various genus $g$. It would be interesting to know if Theorem~\ref{T:M0order2} can be generalized to say something about these potentially more complicated subgroups.

Both authors of this paper were supported in their work on this project by National Science Foundation grant DMS-2005554 ``Smooth $4$--Manifolds: $2$--, $3$--, $5$-- and $6$--Dimensional Perspectives''.

\section{The proof modulo one main calculation}
There are two main ingredients in the proof of Theorem~\ref{T:M0order2}.  The first is part of Lemma~3 in~\cite{GayS4Diffeos}.
\begin{lemma} \label{L:Inverse}
 $\tau_{(S,R)} = \tau_{(R,S)}^{-1}$
\end{lemma}

\begin{proof}
 Switching $R$ and $S$ changes the parametrization of the boundary of a tubular neighborhood of $R \cup S$ from $S^1_l \times S^1_R \times S^1_S$ to $S^1_l \times S^1_S \times S^1_R$. However, we are not changing the orientations of $R$ or $S$, and thus the orientations of the meridians $S^1_R$ and $S^1_S$ do not change. Therefore, to keep our parametrization of the boundary of $R \cup S$ correctly oriented we need to switch the orientation of the longitudinal $S^1_l$ factor. Then, when we turn this into a parametrization of a neighborhood of this $3$--torus as $[-1,1] \times S^1_l \times T^2$, the $[-1,1]$ direction does not change orientation (being oriented by the outward normal convention), and thus the annulus $[-1,1] \times S^1_l$ in fact {\em does} change orientation. Therefore the Dehn twist along this annulus switches from a positive Dehn twist to a negative Dehn twist, and therefore $\tau_{(S,R)} = \tau_{(R,S)}^{-1}$.
\end{proof}

The second ingredient expands on the connection between twists along Montesinos twins and Cerf theoretically described diffeomorphisms of $S^4$ developed in~\cite{GayS4Diffeos}, so as to get a full set of generators and some relations for $\mathcal{M}_0$. We now briefly set up the Cerf theoretic picture in a little more detail. 

Let $\mathop{Emb}(S^1,S^1\times S^3)$ be the space of embeddings of $S^1$ into $S^1 \times S^3$, with basepoint taken to be the embedding $S^1 \times \{p\}$ for some $p \in S^3$. (In other words, without further comment, we will only be working in the component of $\mathop{Emb}(S^1,S^1\times S^3)$ containing $S^1 \times \{p\}$ and we will take that to be the basepoint for $\pi_1(\mathop{Emb}(S^1,S^1\times S^3))$.) There is a homomorphism $\mathcal{H}:\pi_1(\mathop{Emb}(S^1,S^1\times S^3)) \to \pi_0(\mathop{Diff^+}(S^4)$ discussed in~\cite{GayS4Diffeos} (where it is called $\mathcal{FH}_1$) which we describe briefly as follows: Given $a \in \pi_1(\mathop{Emb}(S^1,S^1\times S^3))$, let $\alpha_t :S^1 \into S^1 \times S^3$, with $t \in [0,1]$, be a loop of embeddings representing $a$, with $\alpha_0 = \alpha_1 = S^1 \times \{p\}$. Extend this to a loop of {\em framed} embeddings (the fact that we can do this is also explained in~\cite{GayS4Diffeos}). For each $t$, let $Z_t$ be the $5$--dimensional cobordism built by starting with $[0,1] \times S^4$, attaching a $5$--dimensional $1$--handle along a fixed standard attaching map into $\{1\} \times S^4$, to give an upper boundary canonically identified with $S^1 \times S^3$, and then attaching a $2$--handle along the framed circle $\alpha_t$. Note that $Z_0$ and $Z_1$ are the same $5$--manifold (i.e. built exactly the same way), so we can put these cobordisms together to build a $6$--manifold $W$ fibering over $S^1 = [0,1]/1\sim 0$, with fiber over $t \in S^1$ equal to $Z_t$, so that $W$ itself is a cobordism from $S^1 \times S^4$ to some $4$--manifold bundle over $S^1$. In other words, each $Z_t$ is a cobordism from $S^4$ to some $4$--manifold $X_t$, and the top boundary of $W$ is a $5$--manifold fibering over $S^1$ with fiber over $t \in S^1$ equal to $X_t$. Furthermore, since our basepoint is $S^1 \times \{p\}$, the $2$--handle at $t=0=1$ cancels the $1$--handle, and thus $X_0$ can be canonically identified with $S^4$. Thus the top of the cobordism $W$ is in fact an $S^4$--bundle over $S^1$ with some potentially interesting monodromy which is determined by the (homotopy class of the) loop of attaching maps $\alpha_t$. This monodromy, as an element of $\pi_0(\mathop{Diff^+}(S^4))$, is, by definition, $\mathcal{H}(a)$.

There is an obvious subgroup of $\pi_1(\mathop{Emb}(S^1,S^1\times S^3))$ in the kernel of $\mathcal{H}$, namely the subgroup of homotopy classes represented by embeddings with image equal to $S^1 \times \{p\}$, i.e. the subgroup corresponding to reparametrizations of the domain $S^1$ (or ``spinning the circle in place''). By multiplying by elements of this subgroup, we can thus always assume that our loops of embeddings $\alpha_t: S^1 \into S^1 \times S^3$ have the property that the circle $\{\alpha_t(z) \mid t \in [0,1]\}$, for a fixed $z \in S^1$, is homotopically trivial in $\pi_1(S^1 \times S^3) = \Z$.

The connection between Montesinos twins and loops of circles in $S^1 \times S^3$ is seen as follows: Given a loop $\alpha_t: S^1 \into S^1 \times S^3$ as in the preceding paragraph, suppose that the mapped in torus $T: S^1 \times S^1 \to S^1 \times S^3$ defined by $T(t,z) = \alpha_t(z)$ is actually an embedding. (Budney and Gabai~\cite{BudneyGabai} in fact show that every element of $\pi_1(\mathop{Emb}(S^1,S^1 \times S^3))$ can be represented by such a loop $\alpha_t$.) Let $C$ be the basepoint circle $S^1 \times \{p\}$. Note that $C$ lies on $T$. Surgery along $C$ turns the triple $(S^1 \times S^3,T,C)$ into a triple $(S^4,R,S)$ where $R$ is an embedded $S^2$ in $S^4$ obtained by surgering the torus $T$ along the essential simple closed curve $C$ and $S$ is the embedded $S^2$ which is the core of the surgery (i.e. the surgery replaces $S^1 \times B^3$ with $B^2 \times S^2$, and $S$ is $\{0\} \times S^2 \subset B^2 \times S^2$). Furthermore, $R$ and $S$ intersect transversely at two points, namely the two ``scars'' on $R$ resulting from surgering the torus down to a sphere along $C$. In other words, $W=(R,S)$ is a Montesinos twin in $S^4$, which we will call the twin associated to the loop $\alpha_t$. Note that $S$ is unknotted, since it results from surgery along $S^1 \times \{p\} \subset S^1 \times S^3$, so that $W$ is a half-unknotted Montestinos twin. Conversely, given a Montesinos twin $W=(R,S)$ in $S^4$, {\em if we assume that $S$ is unknotted}, then surgering $(S^4,R,S)$ along $S$ yields $(S^1 \times S^3,T,C)$ where $C = S^1 \times \{p\}$ and $T$ is an embedding of $S^1 \times S^1$ which is the trace of a loop of embeddings $\alpha_t: S^1 \into S^1 \times S^3$, and we will call this the loop of circles associated to the twin $W$.

The following lemma is implicit in the proof of Theorem~4 in~\cite{GayS4Diffeos}.
\begin{lemma} \label{L:TwinsAndTori}
 Let $W=(R,S)$ be a half-unknotted Montesinos twin in $S^4$; for the moment assume that $S$ is unknotted. Let $\alpha_t$ be the loop of circles in $S^1 \times S^3$ associated to $W$. Then $\tau_W = \mathcal{H}([\alpha_t])$. If, on the other hand, $R$ is unknotted, then let $\overline{\alpha_t}$ be the loop of circles associated to $\overline{W} = (S,R)$. In this case, $\tau_W = \tau_{\overline{W}}^{-1} = \mathcal{H}([\overline{\alpha_t}])^{-1} = \mathcal{H}([\overline{\alpha_{1-t}}])$.
\end{lemma}

Note that there are orientation conventions hidden in the above statement; in particular, one needs to understand how the orientations of $R$ and $S$ determine the orientations both of each circle $\alpha_t$, for each $t$, and of the $t$ direction in the loop of circles; equivalently, how the orientations of $R$ and $S$ correspond to the orientations of meridian and longitude for the torus $T: S^1 \times S^1 \into S^1 \times S^3$. The truth is that it suffices to know that there exists some orientation convention that makes it correct, but in the end we will not need to nail down that convention to get our proofs correct because we show that the whole group involved is trivial or cyclic of order two.

\begin{proof}[Proof of Lemma~\ref{L:TwinsAndTori}]
 We have to show that, with appropriate orientation conventions, $\tau_W = \mathcal{H}([\alpha_t])$  when $W=(R,S)$ and $S$ is unknotted. The second half of the statement of Lemma~\ref{L:TwinsAndTori}, for $\overline{W}$, just follows directly from Lemma~\ref{L:Inverse}.
 
 Given a loop of embeddings $\alpha_t: S^1 \into S^1 \times S^3$, a diffeomorphism representing $\mathcal{H}([\alpha_t])$ can be defined as follows. Using the isotopy extension theorem let $\psi_t: S^1 \times S^3 \to S^1 \times S^3$ be an ambient isotopy starting from $\psi_0 = \id$, for $t \in [0,1]$, such that $\alpha_t = \psi_t \circ \alpha_0$. Since $\alpha_1 = \alpha_0 = S^1 \times \{p\} \subset S^1 \times S^3$ we can assume that $\psi_1$ is the identity on $S^1 \times U$ for $U$ a $3$--ball neighborhood of $p$. Then surgery along $\alpha_0 = S^1 \times \{p\}$ turns $S^1 \times S^3$ into $S^4$ in such a way that $\psi_1$ extends as the identity across the surgered region, and thus $\psi_1$ can be seen as a self diffeomorphism of $S^4$, and the isotopy class of $\psi_1$ on $S^4$ is $\mathcal{H}([\alpha_t])$.
 
 When the associated torus $T: S^1 \times S^1 \to S^1 \times S^3$ given by $T(t,b) = \alpha_t(b)$ is an embedding, then there is a standard construction of an explict ambient isotopy $\psi_t$ supported in a neighborhood $D^2 \times S^1 \times S^1$ of $T$ as follows: Let $(r,\theta) \in [0,1] \times [0,2\pi]$ be polar coordinates on $D^2$, with coordinates $(a,b)$ on $S^1 \times S^1$. (We have replaced the original $t$ variable by $a$ because now it represents a spatial coordinate, and $t$ will be used for the time parameter in the isotopy.) Let $f: [0,1] \to [0,1]$ be a smooth nonincreasing function which is $1$ on $[0,\epsilon]$ and $0$ on $[1-\epsilon,1]$ for some suitably small positive $\epsilon$. Then
 \[ \psi_t(r,\theta,a,b) =  (r,\theta,a+tf(r),b)\]
 is the desired isotopy. The complement of $\{r < \epsilon/2\}$ in our neighborhood $D^2 \times S^1 \times S^1$ can now be parameterized (and oriented) as $[\epsilon/2,1] \times S^1_\theta \times S^1_a \times S^1_b$. This orientation agrees with the orientation as $[\epsilon/2,1] \times S^1_a \times S^1_b \times S^1_\theta$ and, with respect to this orientation, we see that $\psi_1$ is a {\em positive} Dehn twist on $[\epsilon/2,1] \times S^1_a$ crossed with the identity on $S^1_b \times S^1_\theta$.
 
 We now need to see that the three circle parameters $S^1_a$, $S^1_b$ and $S^1_\theta$ translate into $S^1_l$, $S^1_R$ and $S^1_S$ when we surger along $C$ to turn $T$ into a Montesinos twin $(R,S)$. In particular we need to make sure that $S^1_a$ becomes $S^1_l$ so that the Dehn twist on $[\epsilon/2,1] \times S^1_a$ becomes the Dehn twist on $[-1,1] \times S^1_l$ in our original parametrization of the neighborhood of a Montesinos twin. First of all, $C$ is the circle $\{r=0,a=0,b \in S^1_b\}$ in $T= \{r=0, a \in S^1_a, b \in S^1_b\}$. The circle $S^1_\theta$ links $T$ and thus, when we surger $T$ to become the $2$--sphere $R$, $S^1_\theta$ becomes the meridian $S^1_R$. The circle $S^1_b$ essentially {\em is} the circle $C$, and thus after surgery becomes the meridian $S^1_S$ to the new $2$--sphere $S$. Finally in order to see that $S^1_a$ becomes the longitudinal circle $S^1_l$, we just need to see that $S^1_a$ is homologically trivial in the complement of $T$. This follows from the fact that $S^1_a$ is homotopically trivial in $S^1 \times S^3$, which we arranged earlier by multiplying by an appropriate element of the domain reparametrization subgroup of $\pi_1(\mathop{Emb}(S^1,S^1\times S^3))$.
\end{proof}

These facts above, together with the fact that $\pi_1(\mathop{Emb}(S^1, S^1 \times S^3))$ is generated by loops which come from embedded tori, immediately gives us the fact that the group of isotopy classes of diffeomorphisms of $S^4$ coming from loops of circles in $S^1 \times S^3$ agrees with the group generated by twists along half-unknotted Montesinos twins:
\begin{corollary} \label{C:M0andHofPi1}
 $\mathcal{H}(\pi_1(\mathop{Emb}(S^1, S^1 \times S^3))) = \mathcal{M}_0$
\end{corollary}
One of the main results of~\cite{GayS4Diffeos} can be restated (combining Corollary~14 and Theorem~4 of~\cite{GayS4Diffeos} with Corollary~\ref{C:M0andHofPi1} above) as:
\begin{theorem}
 $\mathcal{M}_0$ is generated by twists $\tau_{W(i)}$, $i \in \N$, for the Montesinos twins $W(i)=(R(i),S)$ illustrated in Figures~1,~2 and~3 of~\cite{GayS4Diffeos}. The loops of circles $\alpha(i)_t: S^1 \into S^1 \times S^3$ associated to these twins are described by the embedded tori $T(i)$ illustrated in Figure~8 in~\cite{GayS4Diffeos}
\end{theorem}
Figures~\ref{F:snake} and ~\ref{F:FingerMove} reproduce two illustrations of $W(3)$ from~\cite{GayS4Diffeos}; the generalization to $W(i)$ is clear.
\begin{figure}
  \labellist
  \small\hair 2pt
  \pinlabel $R(3)$ [r] at 2 125
  \pinlabel $S$ [l] at 260 106
  \endlabellist
  \centering
  \includegraphics[width=8cm]{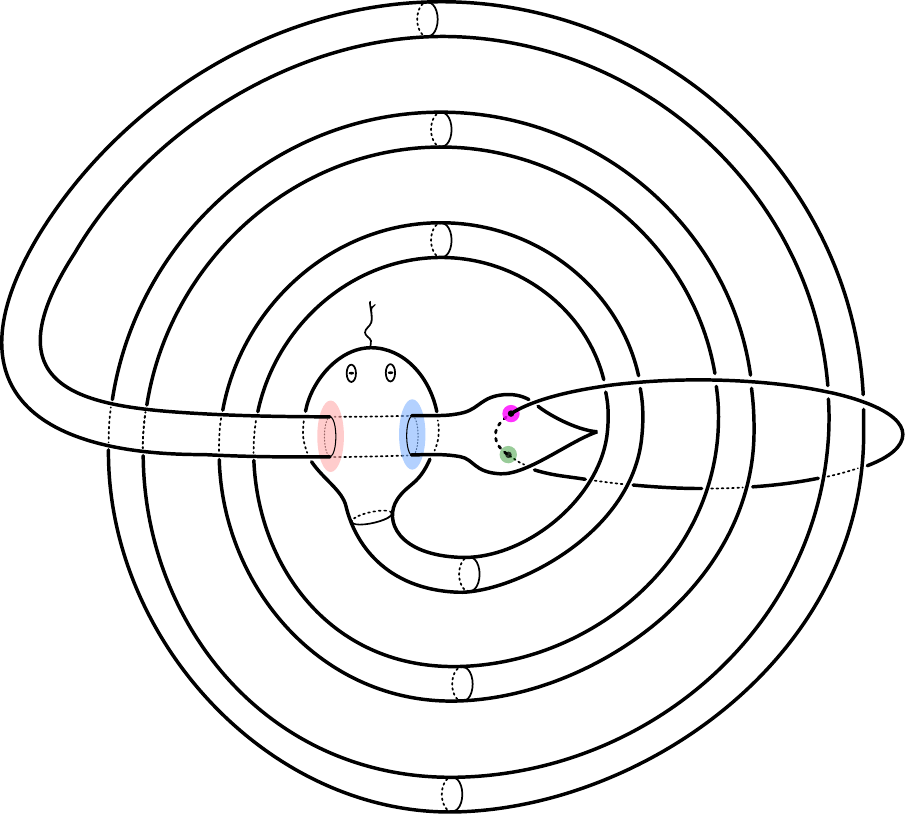}
  \caption{An illustration of $W(3) = (R(3),S)$, with the generalization to $W(i)$ being to wrap $i$ times around instead of three times around. The red and blue disks in $R(3)$ (``ear holes'' of the snake) are pushed forward and backwards in time to avoid self-intersection. We only show the equator of $S$, with the hemispheres lying in the past and future. The two intersection points are colored pink and green.}
  \label{F:snake}
\end{figure}
\begin{figure}
  \labellist
  \small\hair 2pt
  \pinlabel $R(3)$ [r] at 68 69
  \pinlabel $S$ [l] at 124 51
  \endlabellist
  \centering
  \includegraphics{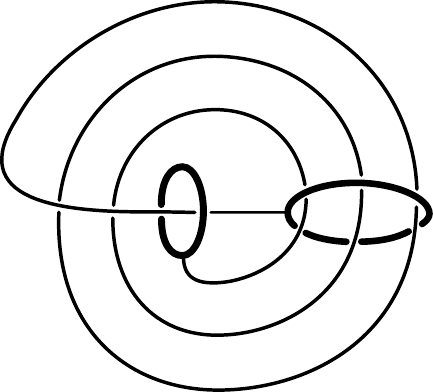}
 \caption{An alternate illustration of $W(3)$, involving two disjoint, embedded $2$--spheres in $S^4$ (the two thick circles capped off with hemispheres in past and future) and an arc connecting them. Pushing a finger from one of the spheres out along this arc and then doing a finger move when one encounters the other sphere, creating a pair of transverse intersections, gives $W(3)$. To recover the illustration in Figure~\ref{F:snake}, push the finger from $R(3)$ until it meets $S$. However, this description is more ``balanced'' between $R(3)$, allowing the user to decide which sphere they prefer to draw as the complicated one.}
 \label{F:FingerMove}
\end{figure}

An important feature of the twins $W(i)=(R(i),S)$ is that {\em both} $R(i)$ and $S$ are unknotted. Thus, in addition to the loops of circles $\alpha(i)_t$ associated to $W(i)$, we have loops of circles $\overline{\alpha(i)_t}$ associated to $\overline{W(i)} = (S,R(i))$. Then by Lemma~\ref{L:Inverse}, we know that
\[ \tau_{W(i)}^{-1} = \mathcal{H}([\overline{\alpha(i)_t}]) \]
Our main calculation in this paper is:
\begin{proposition} \label{P:Relation}
 In $\pi_1(\mathop{Emb}(S^1, S^1 \times S^3))$, $[\overline{\alpha(i)_t}] = n_i [\alpha(1)_t]$ for some integer $n_i$. 
\end{proposition}
(We use additive notation for $\pi_1(\mathop{Emb}(S^1, S^1 \times S^3))$ because Budney and Gabai show that this group is abelian~\cite{BudneyGabai}.) In fact, one can show that $n_i = \pm i$ but we do not need such a precise  result, and the result as stated is quicker and easier to prove. We combine the above result with the following observation:
\begin{lemma} \label{L:W1Flip}
 The twin $W(1)=(R(1),S)$ is isotopic (taking orientations into account) to $\overline{W(1)} = (S,R(1))$ and thus, by Lemma~\ref{L:Inverse}, $\tau_{W(1)} = \tau_{W(1)}^{-1}$.
\end{lemma}

\begin{proof}
 Figure~\ref{F:W1W1bar} illustrates $W(1)$ using the ``finger move'' schematic of Figure~\ref{F:FingerMove} and then shows an isotopy to a diagram which is obviously symmetric between the two $2$--spheres.
\begin{figure}
  \centering
  \includegraphics{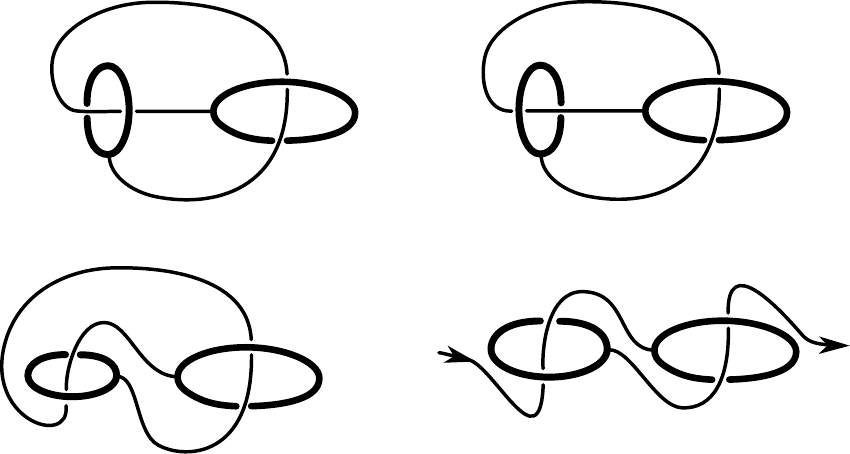}
 \caption{Isotoping $W(1)$ into a symmetric position so as to see that $W(1) = (R(1),S)$ is isotopic to $\overline{W(1)} = (S,R(1))$.}
 \label{F:W1W1bar}
\end{figure}
\end{proof}

Note that in~\cite{BudneyGabai} the authors discuss ``barbell diffeomorphisms", clearly related to Montesinos twists, and give conditions under which barbell diffeomorphisms have order $2$. Most likely Lemma~\ref{L:W1Flip} is a consequence of their Proposition~5.17~\cite{BudneyGabai}.

%
%

From these three results we get:
\begin{proof}[Proof of Theorem~\ref{T:M0order2}]
 An immediate corollary of Proposition~\ref{P:Relation} is that, switching to additive notation for $\pi_0(\mathop{Diff^+}(S^4))$,
 \[ -\tau_{W(i)} = \pm n_i \tau_{W(1)} \]
 From Lemma~\ref{L:W1Flip}, we know that $\tau_{W(1)}$ has order two or is trivial. Since $\mathcal{M}_0$ is generated by $\{\tau_{W(i)}, i \in \N\}$, we conclude that $\mathcal{M}_0$ is generated by $\tau_{W(1)}$ and thus is either the trivial group or the cyclic group of order two. 
\end{proof}

The rest of this paper is devoted to proving Proposition~\ref{P:Relation}

\section{Calculating the Budney-Gabai invariants}

To prove Proposition~\ref{P:Relation}, we need a picture of the loop of circles $\overline{\alpha(i)_t}$ in $S^1 \times S^3$ associated to the Montesinos twin $\overline{W(i)} = (S,R(i))$ in $S^4$. To get this, we need to first draw a picture of $\overline{W(i)} = (S,R(i))$ in which $R(i)$ appears as a standard unknotted $S^2$ and $S$ appears as the interesting half of the twin. Then we can surger along $R(i)$ so as to draw a picture of the resulting embedded torus $\overline{T(i)}$ in $S^1 \times S^3$, from which we can understand the loop of embeddings $\overline{\alpha(i)_t}$. 
This will then be used to compute the $W_2$ invariant of $[\overline{\alpha(i)_t}] \in \pi_1(\mathop{Emb}(S^1, S^1 \times S^3))$ defined in~\cite{BudneyGabai}, which will be sufficient to prove the proposition.

Figure~\ref{F:FingerMove} is most useful for performing the isotopy that standardizes $R(i)$ and leaves $S$ looking complicated. To recover Figure~\ref{F:snake} from Figure~\ref{F:FingerMove}, one pushes the finger from the sphere labelled $R(3)$ along the arc until meeting $S$ and then performing a finger move there. However, one obtains an isotopic Montesinos twin by pushing the figure out along the arc starting from $S$ until meeting $R(3)$ and then performing the finger move; this leaves $R(3)$ still ``looking'' like an unknot. In fact, we can first perform an isotopy to the diagram in Figure~\ref{F:FingerMove} to put $R(3)$ into exactly the position where $S$ was, as in Figure~\ref{F:IsotopyInS4}. The final step in Figure~\ref{F:IsotopyInS4} represents the result of surgering along the (now standardized) $W(3)$ to the torus $\overline{T(3)}$ in $S^1 \times S^3$.
\begin{figure}
\begin{center}

\begin{subfigure}{.3\textwidth}
  \centering
  \fbox{\includegraphics[width=.8\linewidth]{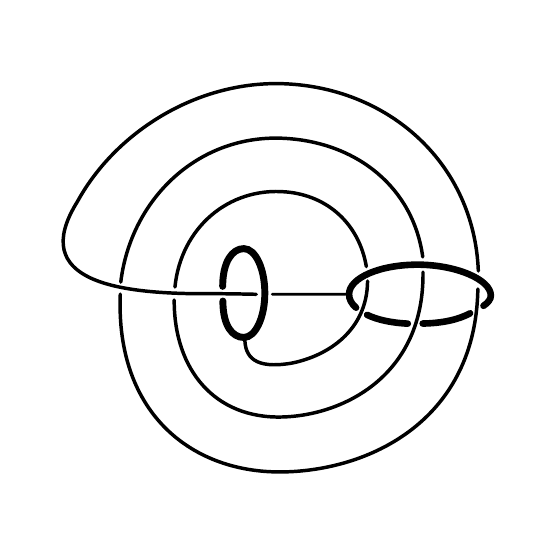}}
\end{subfigure}
\begin{subfigure}{.3\textwidth}
  \centering
  \fbox{\includegraphics[width=.8\linewidth]{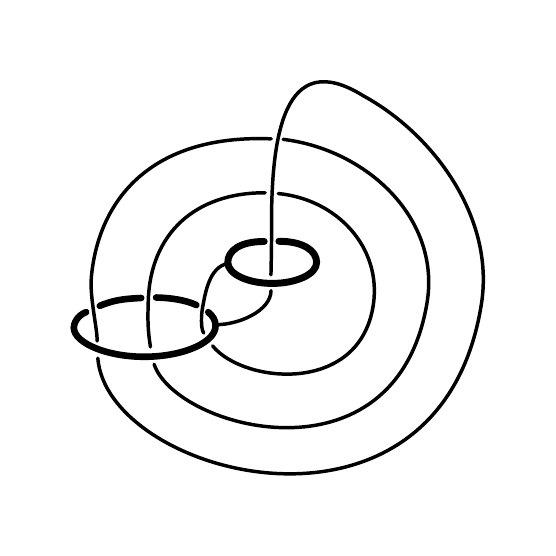}}
\end{subfigure}
\begin{subfigure}{.3\textwidth}
  \centering
  \fbox{\includegraphics[width=.8\linewidth]{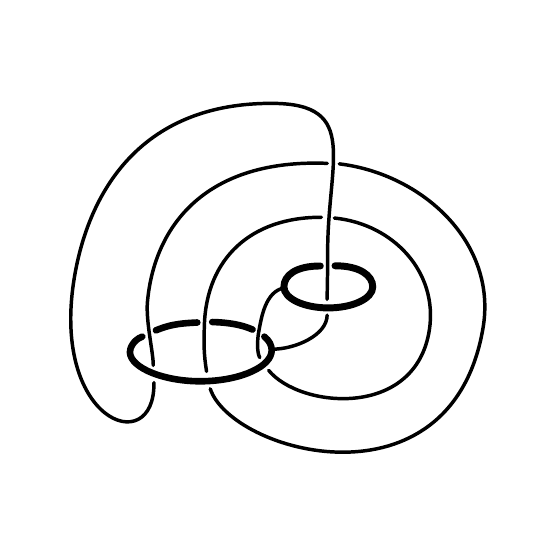}}
\end{subfigure}\\
\begin{subfigure}{.3\textwidth}
  \centering
  \fbox{\includegraphics[width=.8\linewidth]{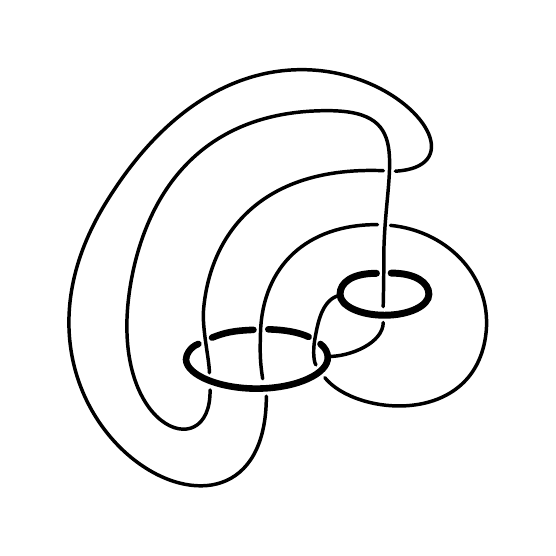}}
\end{subfigure}
\begin{subfigure}{.3\textwidth}
  \centering
  \fbox{\includegraphics[width=.8\linewidth]{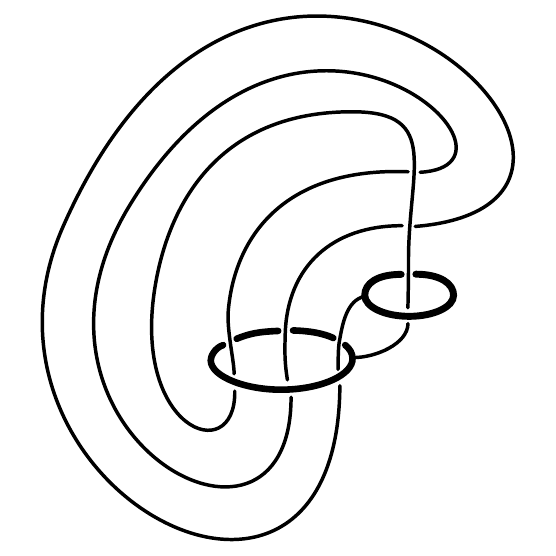}}
\end{subfigure}
\begin{subfigure}{.3\textwidth}
  \centering
  \fbox{\includegraphics[width=.8\linewidth]{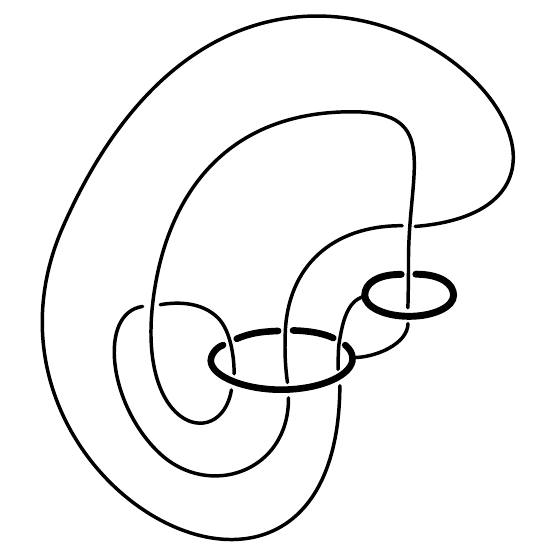}}
\end{subfigure}\\
\begin{subfigure}{.3\textwidth}
  \centering
  \fbox{\includegraphics[width=.8\linewidth]{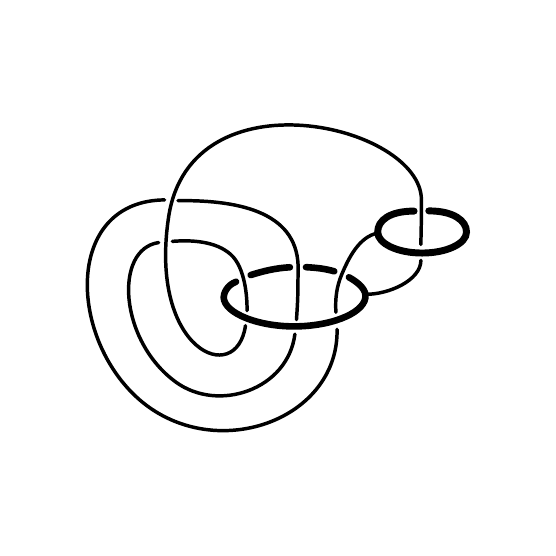}}
\end{subfigure}
\begin{subfigure}{.3\textwidth}
  \centering
  \fbox{\includegraphics[width=.8\linewidth]{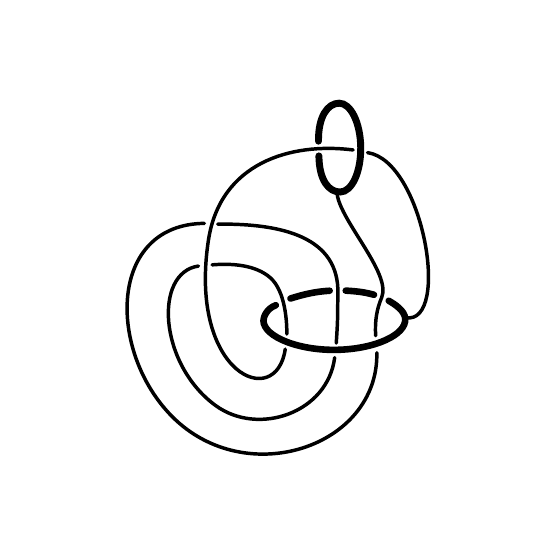}}
\end{subfigure}\\
\begin{subfigure}{.3\textwidth}
  \centering
  \fbox{\includegraphics[width=.8\linewidth]{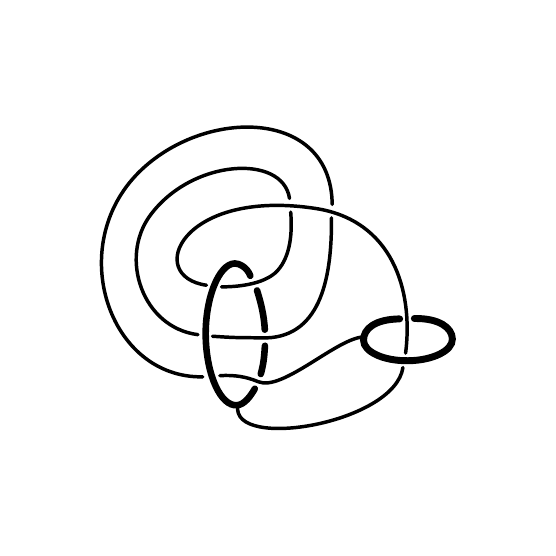}}
\end{subfigure}
\begin{subfigure}{.3\textwidth}
  \centering
  \fbox{\includegraphics[width=.8\linewidth]{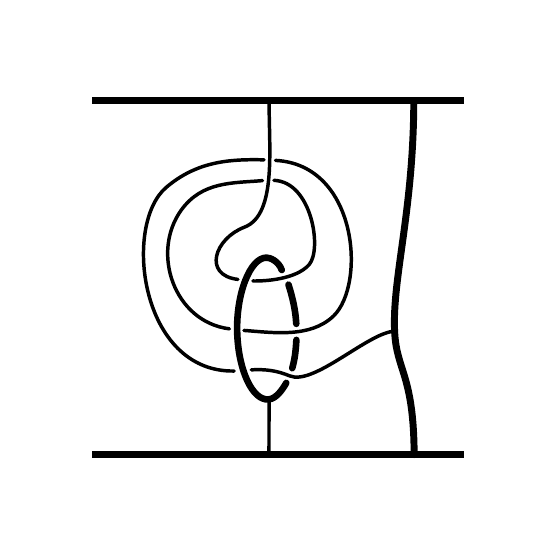}}
\end{subfigure}
 
\end{center}

\caption{An isotopy of $W(3)$. The final frame should be interpreted as a diagram of an embedded torus in $S^1 \times S^3$, the result of surgering along $R(3)$. The interpretation of this diagram is made clearer in Figure~\ref{F:SnakeBarInS1XS3}.}
\label{F:IsotopyInS4}
\end{figure}
Figure~\ref{F:SnakeBarInS1XS3} illustrates this final torus $\overline{T(3)}$ more explicitly in $S^1 \times S^3$, and this should be compared to Figure~\ref{F:FromBudneyGabaiToMontesinos} (Figure~8 in~\cite{GayS4Diffeos}) which illustrates the original $T(3)$.
\begin{figure}
  \labellist
  \small\hair 2pt
  \pinlabel $C$ [r] at 223 130
  \endlabellist
  \centering
  \includegraphics[width=6cm]{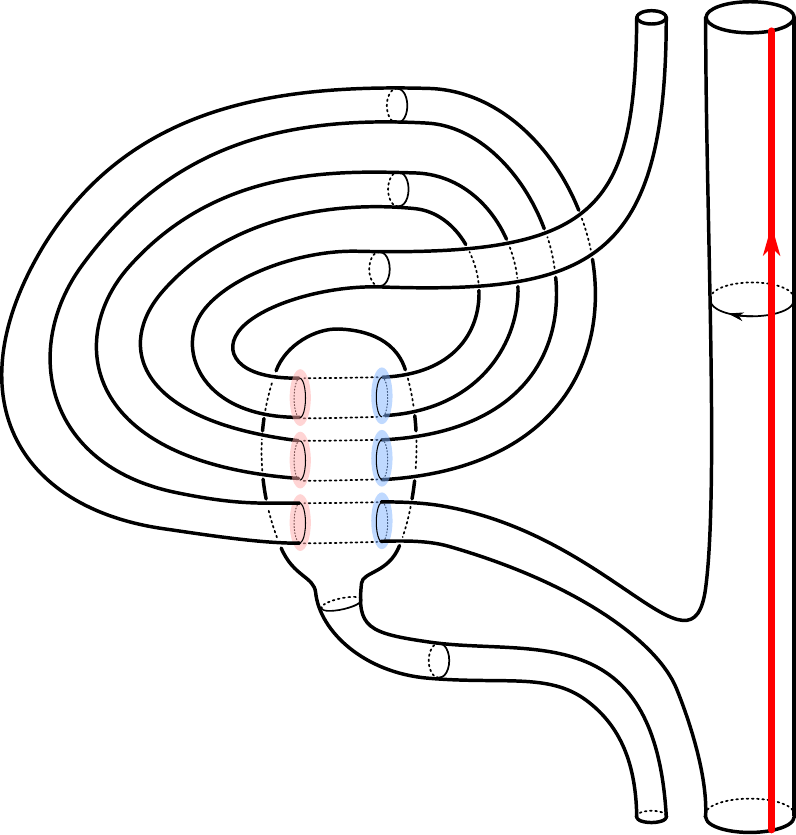}
  \caption{The embedded torus $\overline{T(3)}$ in $S^1 \times S^3$. The top is glued to the bottom, and horizontal slices are $S^3$'s, with the ``time'' coordinate indicated in red/blue shading, as in Figure~\ref{F:snake}.}
  \label{F:SnakeBarInS1XS3}
\end{figure}
\begin{figure}
  \labellist
  \small\hair 2pt
  \pinlabel $C$ [r] at 230 150
  \endlabellist
  \centering
  \includegraphics[width=6cm]{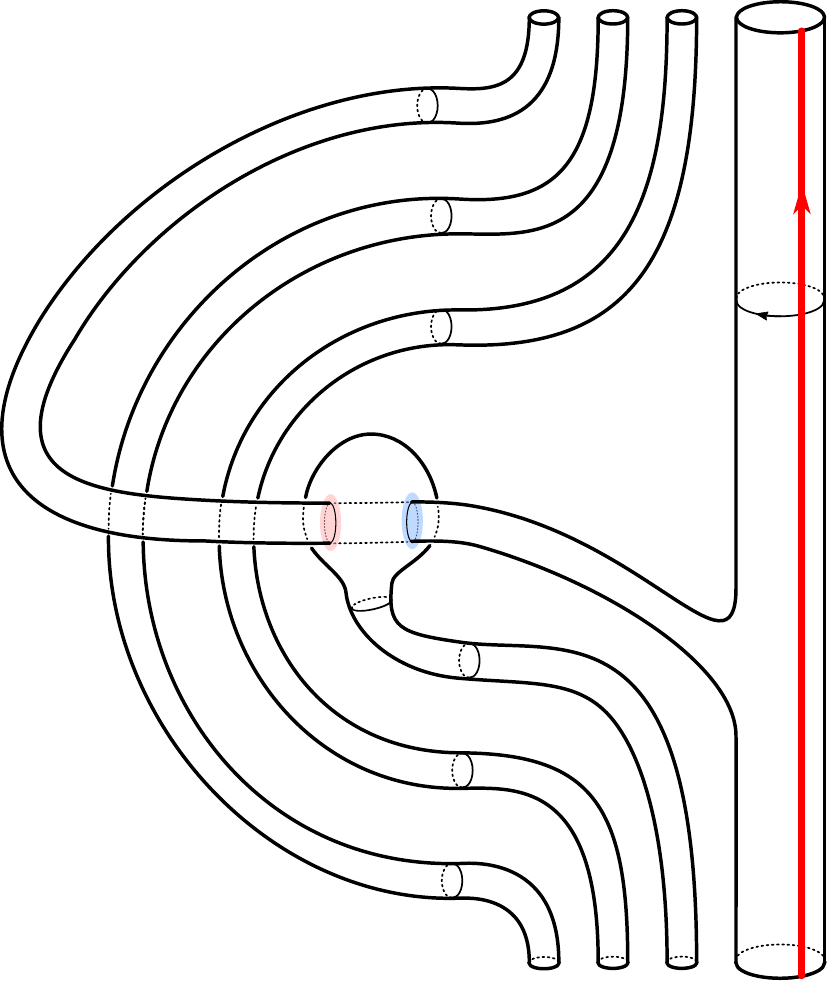}
  \caption{The embedded torus $T(3)$ in $S^1 \times S^3$, the obvious next member of the family of tori described in Figure~4 of~\cite{BudneyGabai}.}
  \label{F:FromBudneyGabaiToMontesinos}
\end{figure}

Now we discuss Budney and Gabai's analysis of $\pi_1(\mathop{Emb}(S^{1},S^{1}\times S^{3}))$ and the aforementioned $W_2$ invariant. 


In Theorem 2.9 of \cite{BudneyGabai}, the authors compute $\pi_1$ of all path components of $\pi_1(\mathop{Emb}(S^{1},S^{1}\times S^{3}))$. Here we summarize their main result only for the component we care about, the component containing our chosen basepoint $S^1 \times \{p\}$. This result is that there is an isomorphism:
\begin{equation}\label{E:Invariants}
W_1\times W_2:\pi_1(\mathop{Emb}(S^{1},S^{1}\times S^{3})) \rightarrow \Z\times \Lambda^{0},
\end{equation}
where
\[
\Lambda^{0} =\Z[x,x^{-1}]/\langle x^{n} - x^{-n}~~\forall n \in \Z, x^{0}, x^{-1} \rangle.
\]
The isomorphism $W_1\times W_2$ is the product of two homomorphisms 
\begin{align*}
W_1:\pi_1(\mathop{Emb}(S^{1},S^{1}\times S^{3}) &\rightarrow \Z\\
W_2:\pi_1(\mathop{Emb}(S^{1},S^{1}\times S^{3}) & \rightarrow \Lambda^{0}.
\end{align*}
The invariant $W_1$ detects "spinning the circle in place" as discussed in the proof of Lemma \ref{L:TwinsAndTori}. We have assumed already that our loops of embeddings $\alpha_t: S^1 \hookrightarrow S^1 \times S^3$ have the property that for a fixed $z\in S^{1}$, the loop $\{\alpha_t(z), t \in [0,1]\}$ is homotopically trivial. This directly translates to saying that we can always assume that $W_1([\alpha_t]) = 0$ for any of the loops of circles we will be considering.  

\begin{remark}
As discussed in Lemma \ref{L:TwinsAndTori}, we are free to reparameterize the domain of our loops, without effecting $\mathcal{H}$. Considering the above isomorphism, \ref{E:Invariants}, we get that $\operatorname{ker}(W_2)\subset \operatorname{ker}(\mathcal{H})$
\end{remark}

The important invariant to discuss is thus $W_2$. We begin by reviewing the definition of $W_2$ found in~\cite{BudneyGabai}. Following the authors' notation, we will denote the two point configuration space of a manifold $M$ by $C_2(M)$. Let $\mathcal{CC}\subset C_{2}(S^{1}\times S^{3})$ be the submanifold of points of the form $((z_1,p),(z_2,p))$, diffeomorphic to $C_{2}(S^{1})\times S^{3}$, with an orientation coming from the diffemorphism,
\[
((z_1,z_2),p) \rightarrow ((z_1,p),(z_2,p)).
\]
Now, given any loop of embeddings $\alpha_t: S^1 \hookrightarrow S^1 \times S^3$ the authors define a map $A:S^{1}\times C_{2}(S^{1}) \rightarrow C_{2}(S^{1}\times S^{3})$ given by the formula
\[
A(t,(z_1,z_2)) = (\alpha_t(z_1),\alpha_t(z_2)).
\]
Assuming that $A$ is transverse to $\mathcal{CC}$, one sees that $A^{-1}(\mathcal{CC})$ is just a finite collection of points. On the set $A^{-1}(\mathcal{CC})$, there is a natural $\Sigma_2$ action given by permuting the $C_2(S^{1})$ coordinates. On the quotient, $A^{-1}(\mathcal{CC})/\Sigma_2$, assign to $[p]\in A^{-1}(\mathcal{CC})/\Sigma_2$ the monomial $\pm x^{k_p}$. Here, the sign of the monomial is given by the signed intersection number of $A(p)$ and $\mathcal{CC}$, for some representative $p$ of $[p]$, and the degree, $k_p$, of the monomial is given by the following procedure. First, we consider the coordinates of the point $p$: $(t,(z_1,z_2))$. Next, take the path $[z_1,z_2]$ going from $z_1$ to $z_2$ in $S^1$ in the positively oriented direction. As $p\in \mathcal{CC}$, the $S^{3}$ coordinate of $\alpha_t(z_1)$ equals the $S^{3}$ coordinate of $\alpha_t(z_2)$. Let $B_p$ denote the arc in $S^1\times S^3$ which connects $\alpha_t(z_2)$ to $\alpha_t(z_1)$ by moving along the $S^1$ factor in the direction opposite to the orientation, while keeping the $S^3$ coordinate fixed. With this, we construct a map $K_p: S^1\rightarrow S^1\times S^3$ by concatenating the arc $\alpha_t([z_1,z_2])$ with the arc $B_p$. The degree, $k_p$ is then given by
\[
k_p = \operatorname{deg}(\pi_{S^1}\circ K_p).
\]
Alternatively, we can calculate $k_p$ first by counting the signed intersection of $K_p$ with $\{z'\}\times S^3$, for a generic choice of $z'\in S^1$. If we choose our $S^3$ slice away from both $\alpha_t(z_1)$ and $\alpha_t(z_2)$, then $k_p$ can be calculated by counting the signed intersection of the arcs $\alpha_t[z_1,z_2]$ and $B_p$ with the $S^3$ slice, and adding the result. So,
\[
k_p = \alpha_t([z_1,z_2])\cdot (\{z'\}\times S^3)+B_p \cdot (\{z'\}\times S^3)
\]
where $\cdot$ is the signed count of transverse intersection points.

Adding the monomials up over all $[p]\in A^{-1}(\mathcal{CC})/\Sigma_2$ gives the formula for $W_2$:
\[
W_2([\alpha_t]) = \sum_{[p]\in A^{-1}(\mathcal{CC})/\Sigma_2}\pm x^{k_p}.
\]
It is shown in \cite{BudneyGabai} that $W_2$ is well defined on homotopy classes of loops when considered as a map to $\Lambda^0$. Moreover, none of the choices made affect the result as long as they are made consistently.

\begin{proof}[Proof of Proposition~\ref{P:Relation}]
We begin by recalling some notation. First, $T(i)$ and $\overline{T(i)}$ are all embedded tori in $S^{1}\times S^{3}$ (see Figures~\ref{F:SnakeBarInS1XS3} and~\ref{F:FromBudneyGabaiToMontesinos}). Each embedded torus then corresponds to a loop of embeddings of circles, $\alpha(i)_t : S^1 \hookrightarrow S^1 \times S^3, t \in [0,1]$ and $\overline{\alpha(i)_t} : S^1 \hookrightarrow S^1 \times S^3, t \in [0,1]$, respectively. Each loop starts at the embedding $t\mapsto (t,p)$, depicted by the red curve $C$ in both Figure~\ref{F:SnakeBarInS1XS3} and~\ref{F:FromBudneyGabaiToMontesinos}. 
Now in \cite{BudneyGabai}, the authors show that $W_2([\alpha(1)_t]) = \pm x^2$. As $W_2$ is a homomorphism, the proposition will follow once we show that $W_2([\overline{\alpha(i)_t}])= n x^2$, for some $n\in \Z$. 

To begin our calculation, we note that the torus $\overline{T(i)}$ can be constructed by tubing together the "standard" torus with a 2--sphere as illustrated in Figure \ref{F:TubingConstruction}, with $i$ equal to the number of times the tube spirals around and through the 2--sphere. The ambient space in the figure should be viewed as $[0,1]\times S^3$, with $\{0\}\times S^3$ being identified with $\{1\}\times S^3$ by the identity. Note that all the "spiralling" the tube does happens between $\{0\}\times S^3$ and $\{1\}\times S^3$. So, to calculate $W_2$, we will compute intersections with the sphere $\{0\}\times S^3$. Note that there are two values of $t$ when $\alpha(i)_t$ intersects $\{0\} \times S^3$ nontransversely, but transversality of $A$ with $\mathcal{CC}$ means that these values of $t$ will not occur in the calculation of $W_2$. 

Now, suppose that $(t,z_1, z_2)$ is a representative for $[p]\in A^{-1}(\mathcal{CC})/\Sigma_2$. Then we have that 
\[
k_p = \overline{\alpha(i)_t}([z_1,z_2])\cdot (\{0\}\times S^3)+B_p \cdot (\{0\}\times S^3)
\]
As $\overline{\alpha(i)_t}([z_1,z_2])$ is a sub arc of the embedded circle $\overline{\alpha(i)_t}(S^1)$, and $\overline{\alpha(i)_t}(S^1)$ has either exactly one positive intersection or exactly two positive intersections and one negative intersection with $\{0\}\times S^3$, the signed intersection count of $\overline{\alpha(i)_t}([z_1,z_2])$ with $\{0\}\times S^3$ lies in the set $\{0, 1, -1, 2\}$. For the arc $B_p$, this is a sub arc of $S^1\times \{v_p\}$ for some point $v_p\in S^3$, oriented opposite to the given orientation of $S^1$ in $S^1 \times S^3$. As such, the intersection number of $B_p$ with $\{0\}\times S^3$ is either $0$ or $-1$. Putting these together, we get that 
\[
W_2([\overline{\alpha(i)_t}]) = m_{-2}x^{-2}+m_{-1}x^{-1}+m_{0}x^{0}+m_{1}x^{1}+m_{2}x^{2}
\]
for some set of integers $m_{k}$, $k = -2,-1,0,1,2$. Finally, by considering the relations for $\Lambda^0$, we get
\[
W_2([\overline{\alpha(i)_t}]) = n x^2
\]
where $n = m_{-2}+m_2$.

\begin{figure}
  \labellist
  \small\hair 2pt
  \endlabellist
  \centering
  \includegraphics[width=8cm]{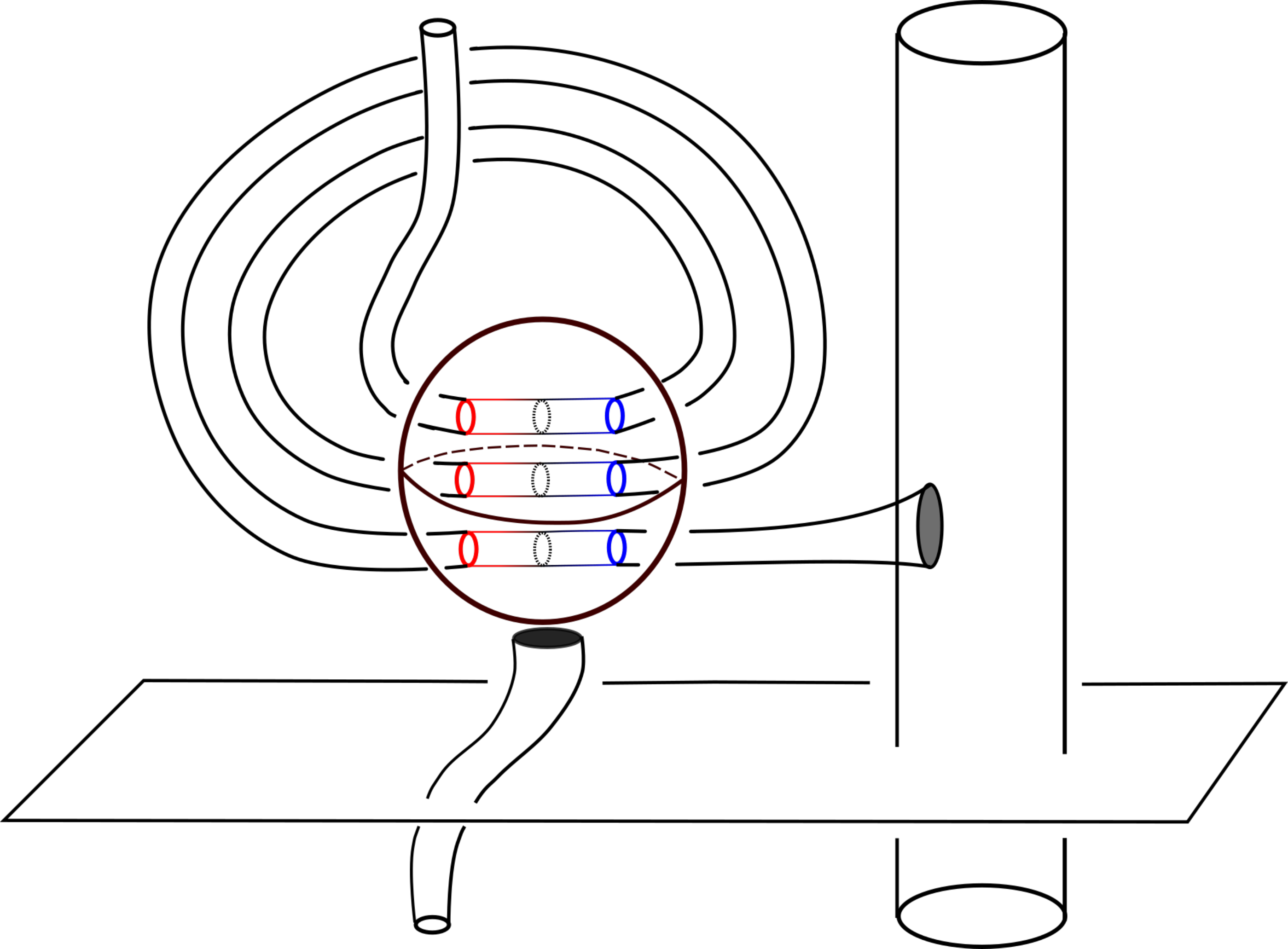}
  \caption{Pictured here is the tubing constructing for the torus $\overline{T(3)}$. The general torus $\overline{T(i)}$ is given by using a tube which links the embedded sphere $i$ times before it is attached to the sphere.  \label{F:TubingConstruction}}
\end{figure}

\end{proof}

\bibliographystyle{plain}
%
%
\bibliography{MontesinosTwistRelations}

\end{document}